\documentclass{amsart}

\usepackage{amsmath, amsthm, amssymb, amsfonts,enumerate}

\usepackage{hyperref}
\usepackage{color}
\usepackage{comment}

\theoremstyle{plain}

\newtheorem{theorem}{Theorem}[section]

\newtheorem{lemma}[theorem]{Lemma}

\newtheorem{corollary}[theorem]{Corollary}

\newtheorem{proposition}[theorem]{Proposition}

\theoremstyle{definition}

\newtheorem{definition}[theorem]{Definition}

\theoremstyle{remark}
\newtheorem{remark}[theorem]{Remark}

\DeclareSymbolFont{AMSb}{U}{msb}{m}{n}
\DeclareMathSymbol{\N}{\mathalpha}{AMSb}{"4E}
\DeclareMathSymbol{\R}{\mathalpha}{AMSb}{"52}
\DeclareMathSymbol{\Z}{\mathalpha}{AMSb}{"5A}
\DeclareMathSymbol{\D}{\mathalpha}{AMSb}{"44}
\DeclareMathSymbol{\s}{\mathalpha}{AMSb}{"53}

\newcommand{\CAT}{CAT}

\newcommand{\eps}{\varepsilon}
\renewcommand{\phi}{\varphi}

\DeclareMathOperator{\Cut}{Cut}

\DeclareMathOperator{\supp}{supp}
\DeclareMathOperator{\Tr}{Tr}
\DeclareMathOperator{\tr}{tr}

\DeclareMathOperator{\Lip}{Lip}
\DeclareMathOperator{\Hess}{Hess}

\DeclareMathOperator{\vol}{vol}

\renewcommand{\tilde}{\widetilde}

\DeclareMathOperator{\diam}{diam}

\newcommand{\ke}{{K}}

\newcommand{\uk}{\kappa}
\DeclareMathOperator{\Ch}{Ch}

\newcommand{\m}{m}
\DeclareMathOperator{\ric}{ric}

\def\S2k{\mathbb{S}^2_\kappa}

\def\co{\colon\thinspace}
\def\SS{\mathbb S}
\def\const{\mathrm{const}}

\author{Vitali Kapovitch}\thanks{University of Toronto, vtk@math.utoronto.ca}
\author{Christian Ketterer}\thanks{University of Toronto, ckettere@math.toronto.edu}
\title{Weakly noncollapsed RCD spaces with upper curvature bounds}
\begin{document}
\begin{abstract}
We show that if a  $CD(K,n)$ space $(X, d, f\mathcal H_n)$ with $n\ge 2$ has curvature bounded above by $\kappa$ in the sense of Alexandrov then  $f=const$.
\end{abstract}
\maketitle
\tableofcontents
\section{Introduction}
In ~\cite{GP-noncol} Gigli and  De Philippis introduced the following notion of a \emph{noncollapsed} $RCD(K,n)$ space. An $RCD(K,n)$ space $(X,d,m)$ is noncollapsed if $n$ is a natural number and $m=\mathcal H_n$.
A similar notion was considered by Kitabeppu in ~\cite{kbg}.

Noncollapsed $RCD(K,n)$ give a natural intrinsic generalization of  noncollapsing limits of manifolds with lower Ricci curvature bounds which are noncollapsed in the above sense by work of
Cheeger--Colding \cite{cheegercoldingI}. 

In ~\cite{GP-noncol} Gigli and  De Philippis also considered the following a-priori weaker notion.  An $RCD(K,n)$ space $(X,d,m)$ is \emph{weakly noncollapsed} if $n$ is a natural number and $m\ll \mathcal H_n$.
Gigli and  De Philippis gave several equivalent characterizations of weakly noncollapsed $RCD(K,n)$ spaces and studied their properties. By work of Gigli--Pasqualetto~\cite{gipa}, Mondino--Kell~\cite{kellmondino} and Bru\'e--Semola~\cite{brusem} it  follows that an $RCD(K,n)$ space is weakly noncollapsed iff $\mathcal R_n\ne\emptyset$ {where $\mathcal{R}_n$ is 
the rectifiable set of $n$-regular points in $X$.}

{
It is well-known that if $(X,d,m)=(M^n,g,e^{-f}d\vol_g)$ where $(M^n,g)$ is a smooth $n$-dimensional Riemannian manifold and $f$ is a smooth function on $M$ then  $(X,d,m)$ is  $RCD(K,n)$ iff $f=\const$. More precisely, 
the classical Bakry-Emery condition $BE(K,N)$, $K\in \mathbb{R}$ and $N\geq n$, for a (compact) smooth metric measure space $(M^n,g,e^{-f}d\vol_g)$, $f\in C^{\infty}(M)$, is
\begin{align*}
\frac{1}{2} L|\nabla u|_g^2\geq \langle \nabla Lu ,\nabla u \rangle_g + \frac{1}{N}(Lu)^2 + K|\nabla u|_g^2, \ \ \forall u\in C^{\infty}(M),
\end{align*}
where $L=\Delta - \nabla f$. In \cite[Proposition 6.2]{saintfleur} Bakry shows that $BE(K,N)$ holds if and only if 
\begin{align*}
\nabla f\otimes \nabla f \leq (N-n) \left(\ric_g + \nabla^2 f- K g\right).
\end{align*}
In particular, if $N=n$, then $f$ is locally constant.

On the other hand, in \cite{erbarkuwadasturm, agsbakryemery} it was proven that a metric measure space $(X,d,m)$ satisfies $RCD(K,N)$ if and only if 
the corresponding Cheeger energy satifies a weak version of $BE(K,N)$ that is equivalent to the classical version for $(M,g,e^{-f}\vol_g)$ from above.}

In ~\cite{GP-noncol} Gigli and  De Philippis conjectured that a weakly noncollapsed $RCD(K,n)$ space is already noncollapsed up to rescaling of the measure by a constant. Our main result is that
this conjecture holds if a weakly noncollapsed space has curvature bounded above in the sense of Alexandrov.

\begin{theorem}\label{main-thm}
Let $n\ge 2$  and let $(X,d,f \mathcal H_n)$ (where $f$ is $L^1_{loc}$ with respect to $\mathcal H_n$ {and $\supp (f\mathcal{H}^n) = X$})  be a complete metric measure space which is $CBA(\uk)$ (has curvature bounded above by $\uk$ in the sense of Alexandrov) and satisfies $CD(K,n)$.
Then $f=\const$ \footnote{{Here and in all applications by $f=\const$ we mean $f=\const$ a.e. with respect to $\mathcal H_n$.}} .
% $\uk (n-1)\geq \ke$,  and $(X,d)$ is an Alexandrov space of curvature bounded below by $\ke-\uk (n-2)$. {In particular $X$ is infinitesimally Hilbertian.}
\end{theorem}

Since smooth Riemannian manifolds locally have curvature bounded above this immediately implies
\begin{corollary}
Let  $(M^n,g)$ be a smooth Riemannian manifold and suppose $(M^n,g,f\mathcal H_n)$ is $CD(K,n)$ where $K$ is finite and $f\ge 0$  is $L^1_{loc}$ with respect to $\mathcal H_n$ {and $\supp (f\mathcal{H}^n)=M$}.
Then  $f=\const$.

\end{corollary}

As was mentioned above, this corollary was well-known in case of smooth $f$  but was not known in case of general locally integrable $f$.

In~\cite{Kap-Ket-18} it was shown that if a $(X,d,m)$ is $CD(K,n)$ and has curvature bounded above then $X$ is $RCD(K,n)$ and if in addition $m=\mathcal H_n$ then $X$ is Alexandrov with two sided curvature bounds.
Combined with Theorem~\ref{main-thm} this implies that the same remains true if the assumption on the measure is weakened to $m\ll \mathcal H_n$.

\begin{corollary}\label{main-cor}
Let $n\ge 2$  and let $(X,d,f \mathcal H_n)$ where $f$  is $L^1_{loc}$ with respect to $\mathcal H_n$ {and $\supp (f\mathcal{H}^n)=X$} be a complete metric measure space which is $CBA(\uk)$ (has curvature bounded above by $\uk$ in the sense of Alexandrov) and satisfies $CD(K,n)$.
Then $X$ is $RCD(K,n)$, $f=\const$, $\uk (n-1)\geq \ke$,  and $(X,d)$ is an Alexandrov space of curvature bounded below by $\ke-\uk (n-2)$.
\end{corollary}
\begin{remark}
Note that since a space $(X,d,f \mathcal H_n)$ satisfying the assumptions of Theorem~\ref{main-thm} is automatically  $RCD(K,n)$, as was remarked in ~\cite{GP-noncol} it follows  from the results of ~\cite{kellmondino} that $n$ must be an integer.
\end{remark}{
Bakry's proof for smooth manifolds does not easily generalize to a non-smooth context. 
But let us describe a strategy that does generalize to $RCD+CAT$ spaces. 

Assume that $(X,d)$ is induced by a smooth manifold $(M^n,g)$ and the density function $f$ is smooth
and positive such that $(X,d,f\m)$ satisfies $RCD(K,n)$. 
Then, by integration by parts on $(M,g)$ the induced Laplace operator $L$ is given by
\begin{align}\label{lll}
L u= \Delta u - \langle \nabla \log f, \nabla u\rangle, \ \ u\in C^{\infty}(M),
\end{align}
where $\Delta u$ is the classical Laplace-Beltrami operator of $(M,g)$ for smooth functions. 
By a recent result of Han one has for any $RCD(K,n)$ space that the operator $L$ is equal to the trace of Gigli's Hessian \cite{giglinonsmooth} on the set of $n$-regular points 
$\mathcal{R}_n$. Hence, after one identifies the trace of Gigli's Hessian with the Laplace-Beltrami operator $\Delta$ of $M$ (what is true on $(M^n,g)$), one obtains immediately that $\nabla \log f=0$. If $M$ is connected, this yields the claim.

The advantage of this approach is that it does not involve the Ricci curvature tensor and in non-smooth context one might follow the same strategy. However, we have to overcome several difficulties that arise from the non-smoothness of the density function $f$ and of the space $(X,d,\m)$.

In particular, we apply the recently developed $DC$-calculus by Lytchak-Nagano for spaces with upper curvature bounds to show that on the regular part of $X$ the Laplace operator with respect to $\mathcal H_n$ is equal to the trace of the Hessian. 
We also show that the combination of CD and CAT condition implies that $f$ is locally semiconcave ~\cite{Kap-Ket-18} and hence locally Lipschitz on the regular part of $X$. This allows us to generalize the above argument for smooth Riemannian manifolds to the general case.

In section 2 we provide necessary preliminaries. We present the setting of $RCD$ spaces and the calculus for them. We state important results by Mondino-Cavalletti
(Theorem \ref{th:laplacecomparison}), Han (Theorem \ref{th:tracelaplace}) and Gigli (Theorem \ref{th:hesscont}, Proposition \ref{prop:Hess}). We also give a brief introduction to the calculus of $BV$ and $DC$ function for spaces with upper curvature bounds. 

In section 3 we develop a structure theory for general $RCD+CAT$ spaces where we adapt the $DC$-calculus of Lytchak-Nagano \cite{Lytchak-Nagano18}. This might be of independent interest. 

Finally, in section 4 we prove our main theorem following the above idea.}
\subsection{Acknowledgements}
The first author is funded by a Discovery grant from NSERC.  
The second author is funded by the Deutsche Forschungsgemeinschaft (DFG, German Research Foundation) -- Projektnummer 396662902.
We are grateful to Alexander Lytchak for a number of  helpful conversations.
\section{Preliminaries}
\subsection{Curvature-dimension condition}
A \textit{metric measure space} is a triple $(X,d,\m)$ where $(X,d)$ is a complete and separable metric space and $\m$ is a locally finite measure.

$\mathcal{P}^2(X)$ denotes the set of Borel probability measures $\mu$ on $(X,d)$ such that $\int_Xd(x_0,x)^2d\mu(x)<\infty$ for some $x_0\in X$ equipped with $L^2$-Wasserstein distance $W_2$. The sub-space of $\m$-absolutely continuous probability measures in $\mathcal{P}^2(X)$ is denoted by $\mathcal{P}^2(X,\m)$.

The \textit{$N$-Renyi entropy} is
\begin{align*}
S_N(\cdot|\m):\mathcal{P}^2_b(X)\rightarrow (-\infty,0],\ \ S_N(\mu|\m)=-\int \rho^{1-\frac{1}{N}}d\m\ \mbox{ if $\mu=\rho\m$, and }0\mbox{ otherwise}.
\end{align*}
$S_N$ is lower semi-continuous, and $S_N(\mu)\geq - \m(\supp\mu)^{\frac{1}{N}}$ by Jensen's inequality.
%{\color{red} Should we also define $CD^*(K,N)$ since it's used in the main theorem?}

For $\kappa\in \mathbb{R}$ we define 
\begin{align*}
\cos_{\kappa}(x)=\begin{cases}
 \cosh (\sqrt{|\kappa|}x) & \mbox{if } \kappa<0\\
1& \mbox{if } \kappa=0\\
\cos (\sqrt{\kappa}x) & \mbox{if } \kappa>0
                \end{cases}
                \quad
   \sin_{\kappa}(x)=\begin{cases}
\frac{ \sinh (\sqrt{|\kappa|}x)}{\sqrt{|\kappa|}} & \mbox{if } \kappa<0\\
x& \mbox{if } \kappa=0\\
\frac{\sin (\sqrt{\kappa}x)}{\sqrt \kappa} & \mbox{if } \kappa>0
                \end{cases}                 
                \end{align*}
Let $\pi_\kappa$ be the diameter of a simply connected space form $\S2k$ of constant curvature $\kappa$, i.e.
\[
\pi_\kappa= \begin{cases}
 \infty \ &\textrm{ if } \kappa\le 0\\
\frac{\pi}{\sqrt \kappa}\ &  \textrm{ if } \kappa> 0

\end{cases}
\]
For $K\in \mathbb{R}$, $N\in (0,\infty)$ and $\theta\geq 0$ we define the \textit{distortion coefficient} as
\begin{align*}
t\in [0,1]\mapsto \sigma_{K,N}^{(t)}(\theta)=\begin{cases}
                                             \frac{\sin_{K/N}(t\theta)}{\sin_{K/N}(\theta)}\ &\mbox{ if } \theta\in [0,\pi_{K/N}),\\
                                             \infty\ & \ \mbox{otherwise}.
                                             \end{cases}
\end{align*}
Note that $\sigma_{K,N}^{(t)}(0)=t$.
For $K\in \mathbb{R}$, $N\in [1,\infty)$ and $\theta\geq 0$ the \textit{modified distortion coefficient} is
\begin{align*}
t\in [0,1]\mapsto \tau_{K,N}^{(t)}(\theta)=\begin{cases}
                                            \theta\cdot\infty \ & \mbox{ if }K>0\mbox{ and }N=1,\\
                                            t^{\frac{1}{N}}\left[\sigma_{K,N-1}^{(t)}(\theta)\right]^{1-\frac{1}{N}}\ & \mbox{ otherwise}.
                                           \end{cases}\end{align*}
\begin{definition}[\cite{stugeo2,lottvillani,bast}]
We say $(X,d,\m)$ satisfies the \textit{curvature-dimension condition} $CD(\ke,N)$ for $\ke\in \mathbb{R}$ and $N\in [1,\infty)$ if for every $\mu_0,\mu_1\in \mathcal{P}_b^2(X,\m)$ 
there exists an $L^2$-Wasserstein geodesic $(\mu_t)_{t\in [0,1]}$ and an optimal coupling $\pi$ between $\mu_0$ and $\mu_1$ such that 
$$
S_N(\mu_t|\m)\leq -\int \left[\tau_{K,N}^{(1-t)}(d(x,y))\rho_0(x)^{-\frac{1}{N}}+\tau_{K,N}^{(t)}(d(x,y))\rho_1(y)^{-\frac{1}{N}}\right]d\pi(x,y)
$$
where $\mu_i=\rho_id\m$, $i=0,1$.

\begin{remark}
If $(X,d,\m)$ is complete and satisfies the condition $CD(\ke,N)$ for $N<\infty$, then $(\supp \m, d)$ is a geodesic space and $(\supp\m,  d,\m)$ is 
$CD(\ke,N)$.

In the following we always assume that $\supp\m=X$.
\end{remark}
\begin{remark} 
For the variants $CD^*(K,N)$ and $CD^e(K,N)$ of the curvature-dimension condition we refer to \cite{bast, erbarkuwadasturm}.
\end{remark}
%
%$(X,d,\m)$ satisfies the \textit{reduced curvature-dimension condition} $CD^*(\ke,N)$ for $\ke\in \mathbb{R}$ and $N\in (0,\infty)$ if we replace 
%in the previous definition the modified distortion coefficients $\tau^{(t)}_{K,N}(\theta)$ by the usual distortion coefficients $\sigma_{K,N}^{(t)}(\theta)$.
%
%If $K=0$, the condition $CD(K,N)$ coincides with the condition $CD^*(K,N)$ and is simply convexity of the $N$-Renyi entropy functional.
\end{definition}
%\begin{remark}
%We note that if a metric measure space $(X,d,\m)$ satisfies a curvature dimension condition $CD(K,N)$, $CD^*(K,N)$ or $CD^e(K,N)$ for $N<\infty$,
%the support $\supp\m$ of $\m$ with the induced metric $d_{\supp\m}$ becomes a geodesic space. 
%This follows since $(\supp\m,d_{\supp\m})$ is complete and a curvature-dimension condition 
%yields that $\supp\m$ is a length space, and is locally compact by
%Bishop-Gromov-type comparison that holds in any case \cite{stugeo2,lottvillani,bast,erbarkuwadasturm}.
%In this paper we will always assume that $\supp m=X$.
%\end{remark}
\subsection{Calculus on metric measure spaces}
%In the following we present the framework for calculus on metric measure spaces by Ambrosio, Gigli and Savar\'e 
For further details about this section we refer to
\cite{agslipschitz,agsheat,agsriemannian,giglistructure}.

Let $(X,d,m)$ be a metric measure space, and let $\Lip(X)$ be the space of Lipschitz functions. 
For $f\in \Lip(X)$ the local slope is
\begin{align*}
\mbox{Lip}(f)(x)=\limsup_{y\rightarrow x}\frac{|f(x)-f(y)|}{d(x,y)}, \ \ x\in X.
\end{align*}
If $f\in L^2(m)$, a function $g\in L^2(\m)$ is called \textit{relaxed gradient} if there exists sequence of Lipschitz functions $f_n$ which $L^2$-converges to $f$, and there exists $h$ such that 
$\mbox{Lip}f_n$ weakly converges to $h$ in $L^2(m)$ and $h\leq g$ $\m$-a.e.\ . $g\in L^2(\m)$ is called the \textit{minimal relaxed gradient} of $f$ and denoted by $|\nabla f|$ if it is a relaxed gradient and minimal w.r.t. the $L^2$-norm amongst all relaxed gradients.
%The \textit{Cheeger energy} $\Ch^X:L^2(m)\rightarrow [0,\infty]$ is defined as 
%\begin{align*}
%2\Ch^X(f)=\liminf_{f_n\in \Lip(X)\overset{L^2(\m)}{\longrightarrow}f}\int_X \mbox{Lip}(f_n)^2 dm.
%\end{align*}
The space of \textit{$L^2$-Sobolev functions} is then $$W^{1,2}(X):= D(\Ch^X):= \left\{ f\in L^2(\m): \int |\nabla f|^2 d\m<\infty\right\}.$$
%For any $f\in W^{1,2}(X)$ there exists a minimal relaxed that is denoted with $|\nabla f|$ and unique up to set of measure $0$. One also calls $|\nabla f|$ the \textit{minimal weak upper gradient} of $f$. Then, one has
%
%\begin{align*}
%\Ch^X(f)=\frac{1}{2}\int |\nabla f|^2 d\m.
%\end{align*}
$W^{1,2}(X)$ equipped with the norm 
$
\left\|f\right\|_{W^{1,2}(X)}^2=\left\|f\right\|^2_{L^2}+\left\||\nabla f|\right\|_{L^2}^2
$
is a Banach space.
If $W^{1,2}(X)$ is a Hilbert space, we say $(X,d,m)$ is \textit{infinitesimally Hilbertian.}
%\begin{remark}
%Note that in general $|\nabla u|\neq \Lip u$ for a Lipschitz function $u$ unless $(X,d,\m)$ satisfies a Poincar\'e inequality and a doubling property \cite[Theorem 5.1]{cheegerlipschitz}.
%By~\cite{rajala1,rajala2} spaces satisfying $CD(K,N)$, $CD^*(K,N)$ or $CD^e(K,N)$ with $N<\infty$ do satisfy a  1-1 Poincar\'e inequality (Theorem \ref{th:rajala}).
%Hence for such spaces \cite{cheegerlipschitz} applies and $|\nabla u|= \Lip u$  a.e. for any Lipschitz $u$.
%\end{remark}
In this case we can define 
\begin{align*}
(f,g)\in W^{1,2}(X)^2\mapsto \langle \nabla f,\nabla g\rangle := \frac{1}{4}|\nabla (f+g)|^2-\frac{1}{4}|\nabla (f-g)|^2\in L^1(m).
\end{align*}
%In particular, the bilinear form $\mathcal{E}(f,g)=\int \langle \nabla f,\nabla g\rangle dm $ becomes a strongly local, quasi-regular Dirichletform.
%We say that $f\in W^{1,2}(X)$ is in the domain of the \textit{Laplace operator $\Delta_X$} if there exists a function $g=:\Delta_X f\in L^2(m)$ such that for every $h\in W^{1,2}(X)$
%\begin{align*}
%\int \langle\nabla f,\nabla h\rangle dm=-\int h\Delta_X f dm.
%\end{align*}
%In this case we say that $f\in D(\Delta_X)$. 
%The vector space $D(\Delta_X)$ is equipped with the operator norm 
%$
%\left\|f\right\|_{D(\Delta_X)}^2=\left\|f\right\|_{L^2}^2+\left\|\Delta_X f\right\|_{L^2}^2.
%$
%Convergence in $D(\Delta_X)$ implies convergence in $W^{1,2}(X)$.
%If $\mathbb{V}$ is any subspace of $L^2(\m)$ and we have that $\Delta f\in \mathbb{V}$, we write $f\in D_{\mathbb{V}}(\Delta)$. $(P_t)_{t\in (0,\infty)}$ denotes the heat semi-group associated to $\Delta$.
Assuming $X$ is locally compact, if $U$ is an open subset of $X$, we say $f\in W^{1,2}(X)$ is in the domain $D({\bf \Delta},U)$ of the \textit{measure valued Laplace} ${\bf \Delta}$ on $U$ if there exists a signed Radon functional ${\bf \Delta}f$ on the set of Lipschitz function $g$ with bounded support in $U$ such that 
\begin{align}\label{equ:integrationbyparts}
\int\langle \nabla g,\nabla f\rangle dm = -\int g d{\bf \Delta}f.
\end{align}
If $U=X$ and ${\bf \Delta}f= [{\bf\Delta} f]_{ac} \m$ with $[{\bf\Delta} f]_{ac}\in L^2(\m)$, we write $[{\bf\Delta} f]_{ac}=:\Delta f$ and $D({\bf \Delta}, X)=D_{L^2(\m)}(\Delta)$. 
$\mu_{ac}$ denotes the $\m$-absolutely continuous part in the Lebesgue decomposition of a Borel measure $\mu$.
If $\mathbb{V}$ is any subspace of $L^2(\m)$ and $f\in D_{L^2(\m)}(\Delta)$ with $\Delta f\in \mathbb{V}$, we write $f\in D_{\mathbb{V}}(\Delta)$. 
\begin{theorem}[Cavalletti-Mondino, \cite{cav-mon-lapl-18}] \label{th:laplacecomparison}
Let $(X,d,\m)$ be an essentially non-branching \linebreak[4] $CD(K,N)$ space for some $K\in \mathbb{R}$ and $N> 1$. For $p\in X$ consider $d_p=d(p,\cdot)$ 
and the associated disintegration $\m=\int_Q h_{\alpha} \mathcal{H}^1|_{X_{\alpha}} q(d\alpha)$.

Then $d_p\in D({\bf \Delta}, X\backslash \left\{p\right\})$ and ${\bf \Delta} d_p|_{X\backslash\left\{p\right\}}$ has the following representation formula:
\begin{align*}
{\bf \Delta} d_p|_{X\backslash\left\{p\right\}}= - (\log h_{\alpha})' \m - \int_Q h_{\alpha}\delta_{a(X_{\alpha})}q(d\alpha).
\end{align*}
Moreover
\begin{align*}
{\bf \Delta} d_p|_{X\backslash\left\{p\right\}}\leq (N-1) \frac{\sin_{K/(N-1)}'(d_p(x))}{\sin_{K/(N-1)}(d_p(x))}\m\ \ \&\ \ 
\left[{\bf \Delta} d_p|_{X\backslash\left\{p\right\}}\right]^{reg}\geq- (N-1) \frac{\sin_{K/(N-1)}'(d_p(x))}{\sin_{K/(N-1)}(d_p(x))}\m.
\end{align*}
\end{theorem}
\begin{remark}\label{rem:disintegration}
The sets $X_{\alpha}$ in the previous disintegration are geodesic segments $[a(X_\alpha),p]$ with initial point $a(X_{\alpha})$ and endpoint $p$.
In particular, the set of points $q\in X$ such that there exists a geodesic connecting $p$ and $q$ that is extendible beyond $q$, is a set of full measure.
\end{remark}
%In analogy to the previous notation we also write $D_{L^2(m)}({\bf \Delta})=D(\Delta)$.
%
%\begin{proposition}[\cite{giglistructure}]\label{prop:chainrulelaplacian}
%Let $(X,d,\m)$ be an infinitesimally Hilbertian metric measure space, let $f\in D({\bf \Delta},U)\cap \Lip(X)$ for an open subset $U\subset X$, let $I\subset \mathbb{R}$ be an open subset, assume $\m(f^{-1}(\mathbb{R}\backslash I))=0$, and let $\phi\in C^2(I)$. 
%Then $\phi\circ f\in D({\bf \Delta},U)$ and 
%\begin{align*}
%{\bf\Delta}(\phi\circ f)= \phi'(f){\bf\Delta} f + \phi''(f)|\nabla f|^2\m \ \mbox{ on }U.
%\end{align*}
%\end{proposition}
\begin{definition}[\cite{agsriemannian,giglistructure}]\label{def:rcd}
A metric measure space $(X,d,m)$ satisfies the Riemannian curvature-dimension condition $RCD(\ke,N)$ for $\ke\in \mathbb{R}$ and {$N\in [1,\infty]$} if it satisfies a curvature-dimension
conditions $CD(\ke,N)$ and is infinitesimally Hilbertian.
\end{definition}
%\begin{definition}
%The space of test functions is 
%\begin{align*}
%\mathbb{D}_{\infty}= D_{W^{1,2}(X)}(\Delta)\cap L^{\infty}(\m)\cap \left\{f\in W^{1,2}(X):|\nabla f|\in L^{\infty}(\m)\right\}.
%\end{align*}
%\end{definition}

In \cite{giglinonsmooth} Gigli introduced a notion of $\Hess f$ in the context of $RCD$ spaces.
$\Hess f$ is tensorial and defined for $f\in W^{2,2}(X)$ that is the second order Sobolev space. An important property of $W^{2,2}(X)$ that we will need in the following is 
\begin{theorem}[Corollary 3.3.9 in \cite{giglinonsmooth}, \cite{savareself}]\label{th:hesscont}
$D_{L^2(\m)}(\Delta)\subset W^{2,2}(X)$.%, and 
%\begin{align*}
%\int |\Hess f|_{HS}^2 d\m\leq \int \left[\left(\Delta f\right)^2 - \ke|\nabla f|^2\right] d\m.
%\end{align*}
\end{theorem}
\begin{remark}\label{rem:H22}
The closure of $D_{L^2(\m)}(\Delta)$ in $W^{2,2}(X)$ is denoted $H^{2,2}(X)$ \cite[Proposition 3.3.18]{giglinonsmooth}. 
\end{remark}
\noindent
The next proposition \cite[Proposition 3.3.22 i)]{giglinonsmooth} allows to compute the $\Hess f$ explicitely.
\begin{proposition}\label{prop:Hess}
Let $f, g_1,g_2\in H^{2,2}(X)$. Then $\langle \nabla f,\nabla g_i\rangle \in W^{1,2}(X)$, and 
\begin{align}\label{id:Hess}
2\Hess f(\nabla g_1,\nabla g_2)= \langle \nabla g_1,\nabla \langle\nabla f,\nabla g_2\rangle \rangle + \langle \nabla g_2,\nabla \langle \nabla f,\nabla g_1 \rangle \rangle+ \langle \nabla f, \nabla \langle \nabla g_1,\nabla g_2\rangle \rangle
\end{align}
holds $\m$-a.e. where the two sides in this expression are well-defined in $L^0(\m)$.
\end{proposition}
%We also recall the following, non straightforward locality property of the Hessian from \cite[Proposition 3.3.24]{giglinonsmooth}.
%\begin{proposition}\label{prop:locality}
%If $f,g\in H^{2,2}(X)$ then $\Hess f=\Hess g$ $\m$-a.e. on $\left\{g=f\right\}$.
%\end{proposition}
\begin{theorem}[\cite{brusem}]
Let $(X,d,\m)$ be a metric measure space satisfying $RCD(\ke,N)$ with $N<\infty$. 
Then, there exist $n\in \mathbb{N}$ and such that set of $n$-regular points $\mathcal{R}_n$ has full measure.
\end{theorem}
\begin{theorem}[\cite{hanriccitensor}]\label{th:tracelaplace} Let $(X,d,\m)$ be as in the previous theorem.
If $N=n\in\mathbb{N}$, then for any $f\in \mathbb{D}_{\infty}$ we have that $\Delta f=\tr\Hess f$ $\m$-a.e.\ . More precisely, if $B\subset \mathcal{R}_n$  is a set of finite measure and
$(e_i)_{i=1,\dots,n}$ is a unit orthogonal basis on $B$, then $$\Delta f|_B=\sum_{i=1}^n\Hess f(e_i,e_i)1_{B}=:[\tr \Hess f]|_B.$$
\end{theorem}\begin{corollary}\label{cor:trace}
Let $(X,d,\m)$ be a metric measure space as before. If $f\in D_{L^2(\m)}(\Delta)$, we have that
$\Delta f=\tr\Hess f$ $\m$-a.e. in the sense of the previous theorem.
\end{corollary}
\subsection{Spaces with upper curvature bounds}

We will assume familiarity with the notion of $CAT(\kappa)$ spaces. We refer to ~\cite{BBI, BH99} or ~\cite{Kap-Ket-18} for the basics of the theory.
\begin{definition}
Given a point $p$ in a $CAT(\kappa)$ space $X$ we say that two unit speed geodesics starting at $p$ define the same direction if the angle between them is zero. This is an equivalence relation by the triangle inequality for angles and the angle induces a metric on the set $S_p^g(X)$ of equivalence classes. The metric completion  $\Sigma_p^gX$ of $S_p^gX$ is called the \emph{space of geodesic directions} at $p$.
The Euclidean cone $C(\Sigma_p^gX)$ is called the \emph{geodesic tangent cone} at $p$ and will be denoted by $T^g_pX$.
\end{definition}
The following theorem is due to Nikolaev~\cite[Theorem 3.19]{BH99}:
\begin{theorem}\label{geod-tangent-cone}
$T_p^gX$ is $CAT(0)$ and $\Sigma_p^gX$ is  $CAT(1)$. 
\end{theorem}
Note that this theorem in particular implies that $T_p^gX$ is a geodesic metric space which is not obvious from the definition.
More precisely, it means  that each path component of $\Sigma_p^gX$ is $CAT(1)$ (and hence geodesic) and the distance between points in different components is $\pi$. Note however, that $\Sigma_p^gX$ itself need not be path connected. 
\begin{comment}
\begin{remark}
We will also need the following simple lemma

\begin{lemma}\label{lem-susp}
Let $(\Sigma, d)$ be $CAT(1)$. Suppose there exist two points $y_+,y_-\in\Sigma$ such that $d(y_+,y_-)=\pi$ and for any point $x\in\Sigma$ it holds that  $d(y_+,x)+d(x,y_-)=\pi$.
Then the midpoint set $\hat \Sigma=\{x: d(x,y_+)=d(x,y_-)=\pi/2$ is convex in $\Sigma$ and $\Sigma$ is isometric to the spherical suspension over  $\hat \Sigma$ with vertices at $y_+$ and $y_-$.
\end{lemma}
\begin{proof}
%This follows by a straightforward application of the rigidity case of Toponogov comparison in $CAT(\kappa)$ spaces.
This lemma is well-known but we sketch a proof for completeness. Let $x,z$ be points in $\Sigma$ with $d(x,y_\pm)=d(z,y_\pm)=\pi/2$ and  $d(x,z)<\pi$ and let $q\in]xz[$.
Then applying ${\eqref{point-on-a-side-comp}}_{(CAT)}$  to the triangles $\Delta(y_\pm, x ,z)$ we get that $d(y_\pm,q)\le d(\bar y_\pm, \bar q)=\pi/2$. Hence  $\pi=d(y_-,q)+d(y_+,q)\le d(\bar y_-, \bar q)+d(\bar y_+, \bar q)=\pi$.

Hence, $d(y_\pm, q)=d(\bar y_\pm, \bar q)=\pi/2$ i.e. we have an equality in  ${\eqref{point-on-a-side-comp}}_{(CAT)}$  in both cases. By the rigidity case of CAT comparison ~\cite[Proposition 2.9]{BH99} this means that $y_-,x,z$ span a (unique) isometric copy of $\Delta(\bar y_-,\bar x,\bar z)$ in $\Sigma$ and the same is true for $y_+$. This easily implies the lemma.

\end{proof}
\end{remark}
\end{comment}

\subsection{$BV$-functions and $DC$-calculus}\label{subsection:BV}
Recall that a function $g: V\subset \mathbb R^n\rightarrow \mathbb R$ of bounded variation ($BV$) admits a derivative in the distributional sense \cite[Theorem 5.1]{Gar-Evans} that is a signed vector valued Radon measure 
$[Dg]=(\frac{\partial g}{\partial x_1}, \dots, \frac{\partial g}{\partial x_n})=[Dg]_{ac}+[Dg]_s$. Moreover, if $g$ is $BV$, then it is $L^1$--differentiable \cite[Theorem 6.1]{Gar-Evans} a.e. with $L^1$-derivative $[Dg]_{ac}$, and approximately differentiable a.e. \cite[Theorem 6.4]{Gar-Evans} with approximate derivative $D^{ap}g=(\frac{\partial^{ap} g}{\partial x_1}, \dots, \frac{\partial^{ap} g}{\partial x_n})$ that coincides almost everywhere with $[Dg]_{ac}$. The set of $BV$-functions $BV(V)$ on $V$ is closed under addition and multiplication \cite[Section 4]{Per-DC}. 
We'll call $BV$ functions $BV_0$ if they are continuous. 
{
\begin{remark}
In ~\cite{Per-DC} and ~\cite{ambrosiobertrand}   $BV$ functions are called  $BV_0$ if they are continuous away from an $\mathcal H_{n-1}$-negligible set.  However, for the purposes of the present paper it will be more convenient to work with the more restrictive definition above.
\end{remark}
}

Then for  $f,g\in BV_0(V)$ we have 
\begin{align}\label{equ:leibniz}
\frac{\partial f g}{\partial x_i}= \frac{\partial f}{\partial x_i} g + f \frac{\partial g}{\partial x_j} 
\end{align}
as signed Radon measures \cite[Section 4, Lemma]{Per-DC}. By taking the $\mathcal{L}^n$-absolutely continuous
part of this equality it follows that  \eqref{equ:leibniz} also holds a.e. in the sense of approximate derivatives. In fact, it  holds at \emph{all} points of approximate differentiability of $f$ and $g$. 
This easily follows by a minor variation of the standard  proof  that $d(fg)=fdg+gdf$ for differentiable functions.
% {\color{red} I don't think that's a formal consequence of this equality as measures.}{\color{blue} Ok, I will write that more carefully or look for a reference.}
% {\color{red} It certainly follows from  \eqref{equ:leibniz}  that this equality holds a.e. as approximate derivatives by taking absolutely continous part of  equality \eqref{equ:leibniz} for measures. However, it's also true that it holds  pointwise  at ALL points of approximate differentiability of $f$ and $g$. 
%This just follows directly because the standard  proof  that $d(fg)=fdg+gdf$ for differentiable functions works for almost differentiable functions pretty much with no changes. This is really quite easy and I wanted to just say it without a proof.
% We don't really need this though and a.e. equality is enough so we can just say that. I'm not sure what is the better option.}
% {\color{blue} Ok, I see, thank you for pointing that out. I guess u mentioned it before, but I forgot it. I would write as u suggest, that is, by definition of approximate differentials the identity holds at every point where $f$ and $g$ are approximate differentiable, or soemthing like that.}
\medskip

A function $f: V\subset \mathbb{R}^n\rightarrow \mathbb{R}$ is called a $DC$--function if in a small neighborhood of each point $x\in V$ one can write $f$ as a difference of two semi-convex functions. The set of $DC$--functions on $V$ is denoted by $DC(V)$ and contains the class $C^{1,1}(V)$. $DC(V)$ is closed under addition and multiplication.
The first partial derivatives $\frac{\partial f}{\partial x_i}$ of a $DC$--function $f: V\rightarrow \mathbb R$ are $BV$, and hence the second partial derivatives $\frac{\partial}{\partial x_j}\frac{\partial f}{\partial x_i}$ exist as signed Radon measure that satisfy $$\frac{\partial}{\partial x_i}\frac{\partial f}{\partial x_j}=\frac{\partial}{\partial x_j}\frac{\partial f}{\partial x_i}$$
\cite[Theorem 6.8]{Gar-Evans}, and hence 
\begin{equation}\label{2n-der-commute}
\frac{\partial^{ap}}{\partial x_i}\frac{\partial f}{\partial x_j}=\frac{\partial^{ap}}{\partial x_j}\frac{\partial f}{\partial x_i}\quad \text{  a.e. on $V$.}
\end{equation}
A map $F: V\rightarrow \mathbb{R}^l$, $l\in\mathbb{N}$, is called a $DC$--map if each coordinate function $F_i$ is $DC$. The composition of two $DC$--maps is again $DC$.
A function $f$ on $V$ is   called $DC_0$ if it's $DC$ and $C^1$.
\medskip

Let $(X,d)$ be a geodesic metric space. A function $f:X\rightarrow \mathbb R$ is called a $DC$-function if it can be locally represented as the difference of two Lipschitz  semi-convex functions. A map $F:Z\rightarrow Y$ between metric spaces $Z$ and $Y$ that is locally Lipschitz is called a $DC$-map if for each 
$DC$-function $f$ that is defined on an open set $U\subset Y$ the composition $f\circ F$ is $DC$ on $F^{-1}(U)$. In particular, a map $F: Z\rightarrow \mathbb R^l$ is $DC$
if and only if its coordinates are $DC$. If $F$ is a bi-Lipschitz homeomorphism and its
inverse is $DC$, we say $F$ is a $DC$-isomorphism.
{
\subsection{$DC$-coordinates in $CAT$-spaces}\label{ln}
The following was developed in \cite{Lytchak-Nagano18} based on previous work by Perelman \cite{Per-DC}.

Assume $(X,d)$ is a $CAT$-space, let $p\in X$ such that there exists an open neighborhood $\hat{U}$ of $p$ that is homeomorphic to $\mathbb{R}^n$. It is well known (see e.g. ~\cite[Lemma 3.1]{Kap-Ket-18} ) that this implies that  geodesics in $\hat{U}$ are locally extendible.

Suppose $T_p^gX\cong \R^n$.

Then,  there exist $DC$ coordinates near $p$ with respect to which the distance on $\hat U$ is induced by a  $BV$ Riemannian metric $g$.

More precisely, let $a_1,\ldots, a_n, b_1,\ldots, b_n$ be points near $p$ such that $d(p,a_i)=d(p,b_i)=r$, $p$ is the midpoint of $[a_i,b_i]$ and $\angle a_ipa_j=\pi/2$ for all $i\ne j$ and all comparison angles $\tilde \angle a_ipa_j, \tilde\angle a_ipb_j,  \tilde\angle b_ipb_j$ are sufficiently close to $\pi/2$ for all $i\ne j$.

Let $x\co \hat U\to \R^n$ be given by $x=(x_1,\dots,x_n)=(d(\cdot,a_1),\ldots, d(\cdot, a_n))$.

Then by  ~\cite[Corollary 11.12]{Lytchak-Nagano18}  for any sufficiently small $0<\eps<\pi_k/4$ the restriction $x|_{B_{2\eps(}p)}$ is Bilipschitz onto an open subset  of $\R^n$. Let $U=B_\eps(p)$ and $V=x(U)$.
By  ~\cite[Proposition 14.4]{Lytchak-Nagano18} $x\co U\to V$ is a DC-equivalence in the sense that $h\co U\to \R$ is DC iff $h\circ x^{-1}$ is DC on $V$.
\medskip

Further, the distance  on $U$  is induced by a $BV$ Riemannian metric $g$ which in $x$ coordinates is given by a $2$-tensor $g^{ij}(p)=\cos \alpha_{ij}$ where 
$\alpha_{ij}$ is the angle at $p$ between geodesics connecting $p$ and $a_i$ and $a_j$ respectively. By the first variation formula $g^{ij}$ is the derivative of $d(a_i, \gamma(t))$ at $0$ where $\gamma$ is the geodesic with $\gamma(0)=p$ and $\gamma(1)=a_j$. Since $d(a_i,\cdot)$, $i=1,\dots n$, are Lipschitz, $g^{ij}$ is in $L^{\infty}$. We denote $\langle v,w\rangle_g(p)= g^{ij}(p)v_iw_j$ the inner product of 
$v,w \in \mathbb R^n$ at $p$. $g^{ij}$ induces a distance function $d_g$ on $V$ such that $x$ is a metric space isomorphism for $\epsilon>0$ sufficiently small.

If $u$ is a Lipschitz function on $U$, $u\circ x^{-1}$ is a Lipschitz function on $V$, and therefore differentiable $\mathcal{L}^n$-a.e. in $V$ by Rademacher's theorem. 
%By slight abuse of notation we will identify $u$ with its coordinate representation. 
Hence,  we can define the gradient of $u$ at points of differentiability of $u$ in the usual way as the metric dual of its differential. Then the usual Riemannian formulas hold and 
$\nabla u=g^{ij}\frac{\partial u}{\partial x_i}\frac{\partial}{\partial x_j}$ and $|\nabla u|^2_g=g^{ij}\frac{\partial u}{\partial x_i}\frac{\partial u}{\partial x_j}$ {a.e.\ .
%where we write $\frac{\partial u}{\partial x_i}(q)= \frac{\partial u\circ x^{-1}}{\partial x_i}(x(q))$.
%In particular, $\nabla x_i=g^{ij}\frac{\partial}{\partial x_j}$ and $|\nabla x_i|_g^2=g^{ii}$. Moreover, identifying $\nabla u$ with $(\frac{\partial u}{\partial x_i})_{i=1,\dots,n}$ one can check that
%$\langle \nabla u,\nabla v\rangle_g= \frac{1}{4}(|\nabla (u+v)|_g^2-\frac{1}{4}|\nabla (u-v)|_g^2$.
}

}
 \section{Structure  theory of RCD+CAT spaces}\label{CD+CAT-str}
%The goal of this section is to provide various sufficient conditions that guarantee that a space satisfying  curvature dimension and $\CAT$ conditions is infinitesimally Hilbertian.

In this section we study  metric measure spaces $(X,d,m)$ satisfying
\begin{equation}
\begin{gathered}\label{eq:cd+cat}
\mbox{$(X,d,m)$ is $\CAT(\uk)$ and satisfies the conditions $RCD(K,N)$ for $1\le N<\infty$, $K,\uk<\infty$.}
\end{gathered}
\end{equation}

The following result was proved in ~\cite{Kap-Ket-18}

\begin{theorem}[\cite{Kap-Ket-18}]\label{CD+CAT implies RCD}
Let $(X,d,\m)$ satisfy $CD(K,N)$ for $1\le N<\infty$, $K,\uk\in \mathbb{R}$. Then $X$ is infinitesimally Hilbertian.
In particular, $(X,d,\m)$ satisfies $RCD(K,N)$.
\end{theorem}
\begin{remark}
It was shown in \cite{Kap-Ket-18} that the above theorem also holds if the $CD(K,N)$ assumption in \eqref{eq:cd+cat} is replaced by $CD^*(K,N)$ or $CD^e(K,N)$ conditions (see \cite{Kap-Ket-18} for the definitions).
Moreover, in a recent paper~\cite{DGPS18} Di Marino, Gigli,  Pasqualetto and Soultanis show that a $CAT(\kappa)$ space with \emph{any} Radon measure  is infinitesimally Hilbertian. For these reasons \eqref{eq:cd+cat} is equivalent to assuming that $X$ is $\CAT(\uk)$ and satisfies one of the assumptions  $CD(K,N), CD^*(K,N)$ or $CD^e(K,N)$ with  $1\le N<\infty$, $K,\uk<\infty$.

\end{remark}
In  \cite{Kap-Ket-18}  we also established the following property of spaces satisfying  \eqref{eq:cd+cat}:
\begin{proposition}[\cite{Kap-Ket-18}]\label{prop:nonbra}
Let $X$ satisfy \eqref{eq:cd+cat}. Then $X$ is non-branching.
\end{proposition}
Next we prove

\begin{proposition}\label{prop:geod-tangent}
Let $X$ satisfy \eqref{eq:cd+cat}.  Then for  almost all $p\in X$ it holds that $T_p^gX\cong \R^k$ for some $k\le N$.
\end{proposition}
\begin{remark}
Note that from the fact that $X$ is an $RCD$ space it follows that $T_pX$ is an Euclidean space for almost all $p\in X$ \cite{gmr}. However, at this point in the proof we don't know if $T_pX\cong T_p^gX$ at all such points (we expect this to be true for all $p$).
\end{remark}

\begin{proof}
%Sketch of proof:
%
First, recall that by the $CAT$ condition, geodesics of length  less than $\pi_\uk$ in $X$ are unique.
Moreover, since $X$ is nonbranching and $CD$, for any $p\in X$ the set $E_p$ of points $q$, such that the geodesic which connects $p$ and $q$ is not extendible, has measure zero (Remark \ref{rem:disintegration}).

Let $A=\{p_i\}_{i=1}^\infty$ be a countable dense set of points in $X$, and let $C=\bigcup_{i\in \mathbb{N}}E_{p_i}$. For any $q\in X\backslash C$ and any $i$  with $d(p_i,q)<\pi_\uk$ the geodesic $[p_iq]$ can be extended slightly past $q$. Since $A$ is dense this implies that for any $q\in X\backslash C$ there is a dense subset in $T_q^gX$ consisting of directions $v$ which have "opposites" (i.e. making angle $\pi$ with $v$).

For every $p\in X$ and every tangent cone $T_pX$ the geodesic tangent cone $T_p^gX$ is naturally a {closed} convex subset of $T_pX$. Since $X$ is $RCD$ this means that for almost all $p$ the geodesic tangent cone $T_p^gX$ is a convex subset of a Euclidean space. Thus, for almost all $p\in X$ it holds that $T_p^gX$ is a convex subset in $\R^m$ for some $m\leq N$,  is a metric cone over $\Sigma_p^gX$ and contains a dense subset of points with opposites also in  $T_p^gX$. 
In particular, $\Sigma_p^gX$ is a convex subset of $\SS^m$. Since a {closed} convex subset of $\SS^m$ is either $\SS^k$ with $k\le m$ or has boundary this means that for any such $p$ $T_p^gX$ is isometric to a Euclidean space  of dimension $k\le m$.
\end{proof}

\begin{proposition}\label{prop:reg-pooints}
Let $X$ satisfy \eqref{eq:cd+cat}. 
\begin{enumerate}[i)]
\item Let $p\in X$ satisfy $T_p^gX\cong \R^m$ for some $m\le N$.

Then an open neighbourhood $W$ of $p$ is homeomorphic to $\R^m$. 
\item If an open neighborhood $W$ of $p$ is homeomorphic to $\R^m$ then for any $q\in W$ it holds that $T_q^gX\cong T_qX\cong  \R^m$.

Moreover, for any compact set $C\subset W$ there is $\eps=\eps(C)>0$ such that every geodesic starting in $C$ can be extended to length at least $\eps$.

\end{enumerate}
\end{proposition}

\begin{proof}
Let us first prove part i).
Suppose $T_p^gX\cong \R^m$. By \cite[Theorem A]{Kramer11} there is a small $R>0$ such that $B_R(p)\backslash \{p\}$ is homotopy equivalent to $\SS^{m-1}$.
Since $\SS^{m-1}$ is not contractible, by~\cite[TRheorem 1.5]{Lyt-Schr07} there is $0<\eps<\pi_\kappa/2$ such that every geodesic starting at $p$ extends to a geodesic of length $\eps$. 
The natural "logarithm" map $\Phi\co \bar B_\eps(p)\to \bar B_\eps(0)\subset T_p^gX$ is Lipschitz since $X$ is $CAT(\kappa)$. By the above mentioned result of Lytchak and Schroeder~\cite[Theorem 1.5] {Lyt-Schr07} $\Phi$ is \emph{onto}.

We also claim that $\Phi$ is 1-1. If $\Phi$ is not 1-1 then 
there exist two distinct unit speed geodesics $\gamma_1,\gamma_2$ of the same length $\eps'\le \eps$ such that $p=\gamma_1(0)=\gamma_2(0)$, $\gamma_1'(0)=\gamma_2'(0)$ but $\gamma_1(\eps')\ne\gamma_2(\eps')$.

Let $v=\gamma_1'(0)=\gamma_2'(0)$. Since $T_p^gX\cong \R^m$ the space of directions  $T_p^gX$ contains the "opposite" vector $-v$. Then there is a geodesic $\gamma_3$ of length $\eps$starting at $p$ in the direction $-v$.
Since  $X$ is $CAT(\kappa)$ {and $2\eps<\pi_k$, the concatenation of $\gamma_3$ with $\gamma_1$ is a geodesic} and the same is true for $\gamma_2$. This  contradicts the fact that $X$ is nonbranching.

Thus, $\Phi$ is a continuous bijection and since both $\bar B_\eps(p)$ and $\bar B_\eps(0)$ are compact and Hausdorff it's a homeomorphism. This proves part i).

Let us now prove part ii). Suppose an open neighborhood $W$ of $p$ is homeomorphic to $\R^m$.  \\By ~\cite[Lemma 3.1]{Kap-Ket-18}  or by the same argument as above using \cite{Kramer11} and ~\cite{Lyt-Schr07}, for any $q\in W$ all geodesics starting at $q$ can be extended to length at least $\eps(q)>0$. Therefore $T_q^gX\cong T_qX$. By the splitting theorem  $T_qX\cong \R^l$ where where $l=l(q)\le N$ might a priori depend on $q$. However,  using part i) we conclude that an open neighbourhood of $q$ is homeomorphic to $\R^{l(q)}$. Since $W$  is homeomorphic to $\R^m$ this can only happen if $l(q)=m$.

The last part of ii) immediately follows from above and compactness of $C$.

\end{proof}
{
\subsection{$DC$-coordinates in $RCD+CAT$-spaces.}

Let $X^g_{reg}$ be the set of points $p$ in $X$ with $T_pX\cong T_p^gX\cong\R^n$. Then by Proposition~\ref{prop:reg-pooints} there is an open  neighbourhood $\hat U$ of $p$ homeomorphic to $\R^n$ such that every $q\in \hat U$ also lies in $X^g_{reg}$. In particular, $X^g_{reg}$ is open. Further, geodesics in $\hat U$ are locally extendible by Proposition~\ref{prop:reg-pooints}. %, for instance again by ~\cite[Theorem 1.5] {Lyt-Schr07}.  

Thus the theory of Lytchak--Nagano from ~\cite{Lytchak-Nagano18} applies, and let $x: U\rightarrow V$ with $U=B_{2\epsilon}(p)\subset \hat{U}$ be $DC$-coordinates as in Subsection \ref{ln}.
The pushforward of the Hausdorff measure $\mathcal{H}^n$ on $U$ under $x$ coordinates is given by $\sqrt{|g|} \mathcal{L}$ where $|g|$ is the determinant of $g_{ij}$ 
Consequently, the map $x\co (U,d,\mathcal H_n)\to (V,d_g,\sqrt{|g|}\mathcal L_n)$ is a metric-measure isomorphism.

With a slight abuse of notations we will identify these metric-measure spaces as well as functions on them, i.e we will identify any function $u$ on $U$ with $u\circ x^{-1}$ on $V$.

%\footnote{We need to say why $\langle \nabla x_i, \nabla x_j\rangle$ in CAT sense coincides with the one in RCD sense. I think it is not obvious, isn t it? One way to argue might be to use the paper of Han-Mondino on angles in RCD spaces. Or is there an easier way?}
\begin{lemma}\label{cont-angles}
Angles between geodesics in $U$ are continuous. That is if $q_i\to q\in U, [q_is_i]\to [qs], [q_it_i]\to [qt]$   are converging sequences  with $q\ne s, q\ne t$  then $\angle s_iq_it_i\to \angle sqt$.
\end{lemma}

\begin{proof}
Without loss of generality we can assume that $q_i\in U$ for all $i$. 
Let   $\alpha_i=\angle s_iq_it_i, \alpha=\angle sqt$. Let $\{\alpha_{i_k}\}$ be a converging subsequence and let $\bar \alpha=\lim_{k\to\infty}\alpha_{i_k}$. Then by upper semicontinuity of angles in $CAT(\kappa)$ spaces it holds that $\alpha\ge \bar\alpha$. We claim that $\alpha = \bar\alpha$.

By Proposition~\ref{prop:reg-pooints} we can extend $[s_iq_i]$ past $q_i$ as geodesics a definite amount $\delta$  to geodesics $[s_iz_i]$.  Let $\beta_i=\angle z_iq_it_i$. By possibly passing to a subsequence of $\{i_k\}$ we can assume that 
$[s_{i_k}z_{i_k}]\to[sz]$. Let $\beta=\angle zqt$. Then since all spaces of directions  $T_{q_i}^gX$ and $T_q^gX$ are Euclidean by Proposition~\ref{prop:reg-pooints},  we have that $\alpha_i+\beta_i=\alpha+\beta=\pi$ for all $i$. Again using semicontinuity of angles we get that $\beta\ge \bar\beta$.

We therefore have
\[
\pi=\alpha+\beta\ge\bar\alpha+\bar\beta=\pi
\]

Hence all the inequalities above are equalities and $\alpha=\bar\alpha$. Since this holds for an arbitrary converging subsequence $\{\alpha_{i_k}\}$ it follows that $\lim_{\i\to\infty} \alpha_i=\alpha$.
\end{proof}

%Recall  that the  map $x\co (U,d,\mathcal H_n)\to (V,d_g,\sqrt{|g|}\mathcal L_n)$ is a metric-measure isomorphism. \medskip

Let $\tilde {\mathcal A}$ be the algebra of functions of  
 the form $\phi(f_1,\ldots, f_m)$ where $f_i=d(\cdot, q_i)$ for some $q_1,\ldots, q_m$ with $|q_ip|>\eps$ and $\phi$ is smooth.  Together with the first variation formula for distance functions  Lemma ~\ref{cont-angles} implies that for any $u,h\in \tilde {\mathcal A}$ it holds that $\langle \nabla u, \nabla h\rangle_g$ is continuous on $V$.  In particular,  $g^{ij}=\langle \nabla x_i,\nabla x_j\rangle_g$ is continuous and hence $g$ is $BV_0$ and not just BV.

Furthermore, since $\frac{\partial}{\partial x_i}=\sum_jg_{ij}\nabla x_j$ where $g_{ij}$ is the pointwise inverse of $g^{ij}$, Lemma ~\ref{cont-angles} also implies that  any $u\in \tilde {\mathcal A}$ is $C^1$ on $V$. Hence, any such $u$ is $DC_0$ on $V$.
\\

Recall that for a Lipschitz function $u$ on $V$ we have two a-priori different notions of the norm of the gradient defined $m$-a.e.: the "Riemannian"  norm of the gradient $|\nabla u|^2_g=g^{ij}\frac{\partial u}{\partial x^i}\frac{\partial u}{\partial x_j}$ and 
 the minimal weak upper gradient $|\nabla u|$  when $u$ is viewed as a  Sobolev functions in $ W^{1,2}(\m)$. We observe that these two notions are equivalent.

\begin{lemma}\label{lem:lipschitz}
Let $u, h: U\rightarrow \mathbb{R}$ be  Lipschitz functions. Then $|\nabla u|=|\nabla u|_g$, $|\nabla h|=|\nabla h|_g$ $\m$-a.e. and $\langle \nabla u,  \nabla h\rangle =\langle \nabla u,  \nabla h\rangle_g$ $\m$-a.e..

In particular, $g^{ij}=\langle \nabla x_i,\nabla x_j\rangle_g=\langle \nabla x_i,\nabla x_j\rangle$ {$\m$-a.e..}
\end{lemma}
\begin{proof}
First note that since both  $\langle \nabla u, \nabla h\rangle$ and $\langle \nabla u,  \nabla h\rangle_g$ satisfy the parallelogram rule, it's enough to prove that $|\nabla u|=|\nabla u|_g$ a.e..

Recall that $g^{ij}$ is continuous on $U$.
Fix a point $p$ where $u$ is differentiable. Then
\begin{align*}
\Lip u(p)= \limsup_{q\rightarrow p}\frac{|u(p)-u(q)|}{d(p,q)}&= \limsup_{q\rightarrow p}\frac{|u(p)-u(q)|}{|p-q|_{g(p)}}\\
&=\sup_{|v|_{g(p)}=1} D_v u=\sup_{|v|_{g(p)}=1} \langle v, \nabla u\rangle_{g(p)}=|\nabla u|_{g(p)}.
%= \lim_{r\rightarrow 0} \sup_{d_g(p,q)=r} \frac{|u(p)-u(q)|}{r}\\
%&= \lim_{r\rightarrow 0} \sup_{\gamma\in \Sigma^g_pX}\frac{|u(\gamma(r))-u(\gamma(0)|}{r}=  \sup_{e\in S^{n-1}(1)}D_eu = |\nabla u|_g.
%\\
%&\geq \lim_{r\rightarrow 0}\sup_{d_g(p,q)=r}\left(\left|D_{(p-q)/|p-q|_2}u \right|+ \frac{o(|p-q|_2)}{|p-q|_2}\right)
%\\
%&= \lim_{r\rightarrow 0} \sup_{|p-q|_{g(p)}=r} \frac{|u(p)-u(q)|}{|p-q|_{g(p)}}
%=\sup_{e\in \mathbb{R}^n, |e|_g=1} |\langle e,\nabla u\rangle_g|= |\nabla u|_g
\end{align*}
%where the second equality holds because $d_g$ is induced by $g$.
%By \cite{cheegerlipschitz} the claim follows.
%=\langle \nabla x_i,\nabla x_j\rangle$.
In the second equality we used that $d$ is induced by $g^{ij}$, and that $g^{ij}$ is continuous.
Since $(U,d,\m)$ admits a local 1-1 Poincar\'e inequality and is doubling, the claim follows from \cite{cheegerlipschitz} where it is proved that for such spaces $\Lip u=|\nabla u|$ a.e..
\end{proof}

In view of the above Lemma from now on we will not distinguish between  $|\nabla u|$ and $ |\nabla u|_g$ and between $\langle \nabla u, \nabla h\rangle$ and $\langle \nabla u,  \nabla h\rangle_g$.

\begin{proposition}\label{prop:BV}
If $u\in W^{1,2}(\m)\cap BV(U)$, then $|\nabla u|^2=g^{ij}\frac{\partial^{ap}u}{\partial x_i}\frac{\partial^{ap}u}{\partial x_j}$ $\m$-a.e.\ .
\end{proposition}
\begin{proof}
We choose a set $S\subset U$ of full measure such that $u$ and $|\nabla u|$ are defined pointwise on $S$ and $u$ is approximately differentiable at every $x\in S$.
Since $u$ is $BV(U)$, for $\eta>0$ there exist $\hat u_\eta\in C^1(U)$ such that for the set
$$
B_\eta=\left\{x\in S: u(x)\neq \hat u_\eta(x), D^{ap}u(x)\neq D\hat u(x)\right\}\cap S
$$
one has $\m(B_\eta)\leq \eta$ \cite[Theorem 6.13]{Gar-Evans}. Note, since $f$ is continuous, there exists a constant $\lambda>0$ such that $\lambda^{-1}\m\leq \mathcal H^n\leq \lambda \m$ on $U$. 
Moreover, since $g^{ij}$ is continuous, one can check that $\hat u_\eta$ is Lipschitz w.r.t. $d_g$, and hence $\hat f\in W^{1,2}(\m)$.

By \cite[Proposition 4.8]{agsheat} we know that $|\nabla u||_{A_\eta}=|\nabla \hat u||_{A_\eta}$ $\m$-a.e. for $A_\eta =S\backslash B_\eta$.
On the other hand, uniqueness of approximative derivatives also yields that $g^{ij}\frac{\partial^{ap}u}{\partial x_i}\frac{\partial^{ap}u}{\partial x_j}|_{A_\eta}=g^{ij}\frac{\partial^{ap}\hat u_\eta}{\partial x_i}\frac{\partial^{ap}\hat u_\eta}{\partial x_j}|_{A_\eta}$ $\m$-a.e.\ . Hence, since $\hat u$ is Lipschitz w.r.t. $d$,
\begin{align*}
|\nabla u| 1_{A_\eta}=g^{ij}\frac{\partial^{ap}u}{\partial x_i}\frac{\partial^{ap}u}{\partial x_j} 1_{A_\eta} \ \ \m\mbox{-a.e. }. 
\end{align*}
by Lemma \ref{lem:lipschitz}. 

Now, we pick a sequence $\eta_k$ for $k\in \mathbb{N}$ such that $\sum_{k=1}^{\infty}\eta_k<\infty$. Then, by the Borel-Cantelli Lemma the set $$B=\left\{x\in S: \exists\mbox{ infinitely many }k\in \mathbb{N}\mbox{ s.t. } x\in B_{\eta_k}\right\}$$ is of $\m$-measure $0$.
Consequently, for $x\in A=S\backslash B$ we can pick a $k\in \mathbb N$ such that $x\in A_{\eta_k}\subset S$. It follows
\begin{align*}
|\nabla u|^2(x)=g^{ij}\frac{\partial^{ap}u}{\partial x_i}\frac{\partial^{ap}u}{\partial x_j}(x) \ \forall x\in S 
\end{align*}
and hence $\m$-a.e.\ .
\end{proof}}

\section{Proof of the main theorem}
{\bf 0.} Let $(X,d, f\mathcal H_n)$ be $RCD(\ke,n)$ and $\CAT(\uk)$ where $0\le f\in L^1_{loc}(\mathcal{H}^n)$.

\begin{remark}
If $(X,d,\m)$ is a weakly non-collapsed $RCD$-space in the sense of \cite{GP-noncol} or a space satisfying  the generalized Bishop inequality in the sense of \cite{kbg} and if $(X,d)$ is $\CAT(\uk)$, the assumptions are satisfied by \cite[Theorem 1.10]{GP-noncol}.
\end{remark}
Following Gigli and  De Philippis~\cite{GP-noncol} for any $x\in X$ we consider the monotone quantity $\frac{m(B_r(x))}{v_{k,n}(r)}$ which is non increasing in $r$ by the Bishop-Gromov volume comparison. 
Let $\theta_{n,r}(x)=\frac{m(B_r(x))}{\omega_nr^n}$. Consider the density function $\theta_{n}(x)=\lim_{r\to 0}\theta_{n,r}(x)=\lim_{r\to 0}\frac{m(B_r(x))}{\omega_nr^n}$.

 Since $n$ is fixed throughout the proof we will drop the  subscripts $n$  and from now on use the notations $\theta(x)$ and $\theta_{r}(x)$ for $\theta_{n}(x)$ and $\theta_{n,r}(x)$ respectively.

By Propositions~\ref{prop:geod-tangent}, \ref{prop:reg-pooints} and \cite[Theorem 1.10]{GP-noncol} we have that for almost all $p\in X$ it holds that $T_pX\cong T_p^gX\cong\R^n$ and $\theta(x)=f(x)$.

Therefore we can and will assume from now on that $f=\theta$ everywhere.
 
\begin{remark}
Monotonicity of  $r\mapsto \frac{m(B_r(x))}{v_{k,n}(r)}$ immediately implies that $f(x)=\theta(x)>0$ for all $x$.
 \end{remark}
 
 Let $x\in X_{reg}^g$. Then $T_p^gX\cong \R^m$ for some $m\le n$. We claim that $m=n$. By  Proposition~\ref{prop:reg-pooints}  $X_{reg}^g$ is an $m$-manifold near $p$ and by section~\ref{ln} DC coordinates near $p$ give a \emph{biLipschitz} homeomorphism of an open neighborhood of $p$ onto an open set in $\R^m$. Since $m=f\mathcal H_n$ this can only happen if $m=n$.

\begin{lemma}
\label{dens-semiconcave}\cite[Lemma 5.4]{Kap-Ket-18} $\theta=f$ is  semiconcave on $X$.
\end{lemma}
\begin{corollary}
$\theta=f$ is locally Lipschitz near any {$p\in X_{reg}^g$.} %with $T_p^gX\cong \R^m$.
\end{corollary}
\begin{proof}
First observe that semiconcavity of $\theta$, %Bishop-Gromov relative volume comparison 
the fact that $\theta \ge 0$ and local extendability of geodesics on $X_{reg}^g$
imply that $\theta$ must be locally bounded on $X_{reg}^g$.
Now the corollary becomes an easy consequence of Lemma~\ref{dens-semiconcave}, the fact that geodesics are locally extendible a definite amount near $p$ by Proposition~\ref{prop:reg-pooints}
and the fact that a semiconcave function on $(0,1)$ is locally Lipschitz.
\end{proof}
%\begin{remark}
%We do not know if a continuous semiconcave function on an $RCD(\uk, N)$ space is automatically locally Lipschitz. This is known to be true for Alexandrov spaces.

%\end{remark}
%\noindent
%Let $(X,d, f\mathcal H_n)$ be $RCD(\ke,n)$ and $\CAT(\uk)$. 
%
{\bf 1.} Since small balls in spaces with curvature bounded above are geodesically convex, 
we can assume that $\diam X<\pi_\kappa$. { Let $p\in X$, $x:U\rightarrow \mathbb{R}^n$ and $\tilde{\mathcal{A}}$ be as in the previous subsection.

By the same argument as in ~\cite[Section 4]{Per-DC} (cf. \cite{palvs}, \cite{ambrosiobertrand})  it follows that any $u\in \tilde {\mathcal A}$ lies in  $D({\bf \Delta},U,\mathcal H_n)$  and the $\mathcal{H}^n$-absolutely continuous part of ${\bf \Delta}_0 u$ can be computed using standard Riemannian geometry formulas that is 
\begin{align}\label{equ:laplace}
{\bf\Delta}_0^a(u)=\frac{1}{\sqrt{|g|}}\frac{\partial^{ap}}{\partial x_j}\bigl( g^{jk}\sqrt{|g|}\frac{\partial u}{\partial x_k}\bigr)
\end{align}
where $|g|$ denotes the pointwise determinant of $g^{ij}$.  Here ${\bf \Delta} _0$ denotes the measure valued Laplacian on $(U, d, \mathcal H_n)$.
Note that $g$, $\sqrt{|g|}$ and $\frac{\partial u}{\partial x_i}$ are $BV_0$-functions, and the
derivatives on the right are understood as approximate derivatives.

Indeed, w.l.o.g. let $u\in DC_0(U)$, and let $v$ be Lipschitz with compact support in $U$. As before we identify $u$ and $v$ with their representatives in $x$ coordinates. 
First, we note that, since $g$, $\sqrt{|g|}$ and $\frac{\partial u}{\partial x_i}$ are $BV_0$, their product is also in $BV_0$, as well as the product with $v$. Then, the Leibniz rule \eqref{equ:leibniz} for the approximate partial derivatives yieds that 
\begin{align*}
\sqrt{|g|}g^{ij}\frac{\partial u}{\partial x_i}\frac{\partial^{ap} v}{\partial x_j}=- \frac{\partial^{ap}}{\partial x_j}\left(\sqrt{|g|}g^{ij}\frac{\partial u}{\partial x_i}\right) v +\frac{\partial^{ap}}{\partial x_j}\left(\sqrt{|g|}g^{ij}\frac{\partial u}{\partial x_i} v\right) \ \mathcal{L}^n\mbox{-a.e.}\ .
\end{align*}
Again using  \eqref{equ:leibniz} we  also have that

\begin{align}
\sqrt{|g|}g^{ij}\frac{\partial u}{\partial x_i}\frac{\partial v}{\partial x_j}=- \frac{\partial}{\partial x_j}\left(\sqrt{|g|}g^{ij}\frac{\partial u}{\partial x_i}\right) v +\frac{\partial }{\partial x_j}\left(\sqrt{|g|}g^{ij}\frac{\partial u}{\partial x_i} v\right) \mbox{as measures}\ 
\end{align}
and the absolutely continuous with respect to  $\mathcal{L}^n$ part of this equation is given by the previous identity.

The fundamental theorem of calculus for BV functions (see ~\cite[Theorem 5.6]{Gar-Evans}) yields that 
\begin{align}\int_V \frac{\partial }{\partial x_j}\left(\sqrt{|g|}g^{ij}\frac{\partial u}{\partial x_i} v\right) =0. \end{align}

Moreover, by Lemma \ref{lem:lipschitz} $\langle \nabla v,\nabla u\rangle$ is given in $x$ coordinates by $g^{ij}\frac{\partial v}{\partial x_j}\frac{\partial u}{\partial x_i} $ $\mathcal{L}^n$-a.e.\ .

Combining the above formulas gives that

\begin{align*}
%\int_U v\Delta_0 u d\mathcal{H}^n=
-\int_V \langle \nabla u,\nabla v\rangle \sqrt{|g|} d\mathcal{L}^n&= \int_V \left[ \frac{\partial^{ap}}{\partial x_j}\left(
\sqrt{|g|}g^{ij}\frac{\partial u}{\partial x_i}\right) v -\frac{\partial^{ap}}{\partial x_j}\left(
\sqrt{|g|}g^{ij}\frac{\partial u}{\partial x_i} v\right)\right] d\mathcal{L}^n\\
&= \int_V \frac{1}{\sqrt{|g|}} \frac{\partial^{ap}}{\partial x_j}\left(
\sqrt{|g|}g^{ij}\frac{\partial u}{\partial x_i}\right) v \sqrt{|g|} d\mathcal{L}^n+\int_Vv d\mu
\end{align*}
where $\mu$ is some signed measure such that $\mu\perp \mathcal L^n$.
%and the right hand side is integrable w.r.t. $\mathcal L^n$. 
This implies \eqref{equ:laplace}.
%
%The derivatives on the right are understood in the distributional sense or in the $L^1$ sense or as approximate derivatives. Recall that a  BV function $h$ on an open set in $\R^n$ is almost everywhere $L^1$ and approximately differentiable  and both approximate and $L^1$ derivatives of $h$ are almost everywhere equal to the absolutely continuous part of its distributional derivative~\cite[Theorem 6.1,6.4]{Gar-Evans}.

\medskip
{\bf 2.} Since $(X,d,m)$ is $RCD(K,n)$ for any $q\in X$, we have that $d_q$ lies in $D({\bf \Delta},U\backslash \{q\},m)$  and  ${\bf \Delta}d_q$ is locally bounded above on $ U\backslash \{q\}$ by $const\cdot m$ by Theorem \ref{th:laplacecomparison}.

Furthermore, since by Proposition~\ref{prop:reg-pooints} all geodesics in $U$ are locally extendible  %(compare with the first paragraph in {\bf 1.})  
we have ${\bf \Delta}d_q= \left[{\bf \Delta} d_q\right]^{reg} \cdot\m$ on $U\backslash \left\{q\right\}$ and $\left[{\bf \Delta} d_q\right]^{reg} $ is locally bounded below on $ U\backslash \{q\}$ again by Theorem \ref{th:laplacecomparison}. Therefore $\left[{\bf \Delta} d_q\right]^{reg} $ is in $L^\infty_{loc}(U\backslash \{q\})$ with respect to $m$ (and also $\mathcal H_n)$, and in particular, ${\bf \Delta} d_q$ is locally $L^2$.  

By the chain rule for ${\bf \Delta}$ \cite{giglistructure} the same holds for any $u,h\in \tilde {\mathcal A}$ on all of $U$ as by construction $u$ and $h$ only involve distance functions to points outside $U$.

Recall the following lemma from \cite[Lemma 6.7]{amslocal} (see also \cite{mondinonaber}).
\begin{lemma}\label{lem:cutoff}
Let $(X,d,\m)$ be a metric measure space satisfying a $RCD$-condition. Then for all $E\subset X$ compact and all $G\subset X$ open such that $E\subset G$ there exists a Lipschitz function $\chi:X\rightarrow [0,1]$ with
\begin{itemize}
 \item[(i)] $\chi=1$ on $E_\nu=\left\{x\in X:\exists y\in E: d(x,y)<\nu\right\}$ and $\supp\chi\subset G$,
 \medskip
 \item[(ii)] ${\bf\Delta}\chi\in L^{\infty}(\m)$ and $|\nabla\chi|^2\in W^{1,2}(X)$.
\end{itemize}
\end{lemma}
%\begin{remark}
%Following the proof of this Lemma  in \cite{amslocal} we see that one can choose $\chi$ to be in $\mathbb{D}_{\infty}^X$.
%\end{remark}
Let us choose a cut-off function $\chi: X\rightarrow [0,1]$ as in the previous lemma for $G$ with $\bar{G}\subset U$ and $E=\bar{B}_{\delta}(p) \subset G$ for some $\delta \in (0,\eps)$.

{Let $u,h\in \tilde {\mathcal A}$.}
By the chain rule for  ${\bf \Delta}$ it again follows that $$ {\bf \Delta}(\chi u)=\left[ {\bf \Delta}(\chi u)\right]^{reg} \m \ \ \& \ \ \left[ {\bf \Delta}(\chi u)\right]^{reg} \in L^2(\m).$$
Moreover, (\ref{equ:integrationbyparts}) holds for Lipschitz functions on $X$. Hence $\chi u\in D_{L^2(\m)}(\Delta)$.

Therefore $\chi u,\chi h\in H^{2,2}(X)$ by Remark \ref{rem:H22}, $\langle \nabla \chi u,\nabla \chi h\rangle\in W^{1,2}(U)$ by Proposition \ref{prop:Hess} and the Hessian of $\chi u$ can be computed by the formula (\ref{id:Hess}). 
Moreover, by locality of the minimal weak upper gradient
%Hence, by locality of $\Hess$ for any $u,h_1,h_2\in  \tilde {\mathcal A}$ (Proposition \ref{prop:locality}) the identity
\begin{align}\label{hess-1}
&2\Hess (\chi u)(\nabla (\chi h_1), \nabla (\chi h_2))|_{B_{\delta}(p)}\nonumber\\
&\ \ \ \ \ \ = \langle \nabla h_1,\nabla\langle\nabla u,\nabla h_2\rangle\rangle + \langle \nabla h_2,\nabla \langle \nabla u,\nabla h_1\rangle \rangle - \langle \nabla u,\nabla \langle\nabla h_1,\nabla h_2\rangle\rangle\ \mbox{$\m$-a.e. in $\bar{B}_{\delta}(p)$}.
\end{align}
Note that, for instance, $$W^{1,2}(\bar{B}_{\delta}(p))\ni \langle \nabla \chi u, \nabla \chi h_2\rangle|_{\bar{B}_{\delta}(p)}= \langle \nabla \chi u|_{\bar{B}_{\delta}(p)}, \nabla \chi h_2|_{\bar{B}_{\delta}(p)}\rangle =  \langle \nabla u|_{\bar{B}_{\delta}(p)}, \nabla h_2|_{\bar{B}_{\delta}(p)}\rangle.$$
\begin{remark} 
It is not clear that $u$ itself is in the domain of Gigli's Hessian since $u$ is not contained $D_{L^2(\m)}(\Delta)$ (integration by parts for $u$ would involve boundary terms).
Nevertheless, the equality and the RHS in \eqref{hess-1} are well-defined on $B_{\delta}(p)$.
We denote the RHS in \eqref{hess-1} with $Hu(h_1,h_2)$.
\end{remark}
\medskip

{\bf 3.} The aim of this paragraph is to compute $H u( x_i, x_j)g_{ij}$ on $B_{\delta}(p)$ in the $DC_0$ coordinate chart $x$.
In the following we assume w.l.o.g. that $B_{\eps}(p)=B_{\delta}(p)$ for $\delta$ like in the previous paragraph.

Since $u,h_1,h_2$ are $DC_0$ in $x$ coordinates we have that $\langle \nabla h_1,\nabla h_2\rangle= g^{ij}\frac{\partial h_1}{\partial x_i}\frac{\partial h_2}{\partial x_j}$ is $BV$  and the same holds for $\langle \nabla u,\nabla h_1\rangle$, $\langle \nabla u,\nabla h_2\rangle$.
Moreover, $\langle \nabla h_i,\nabla h_j\rangle, \langle \nabla u,\nabla h_i\rangle\in W^{1,2}(\bar{B}_{\delta}(p))$ as we saw before. 

Hence, with the help of Proposition \ref{prop:BV} the RHS of \eqref{hess-1} can be computed pointwise in $x$ coordinates at points of approximate differentiability of $\frac{\partial u}{\partial x_i}, \frac{\partial h_1}{\partial x_i}$ and $\frac{\partial h_2}{\partial x_i}$, $i=1,\dots n$, and \eqref{hess-1} can be understood to hold a.e. in the sense of approximate derivatives.
That is, we can write
\begin{equation}\label{hess-2}
\langle \nabla u,\nabla \langle\nabla h_1,\nabla h_2\rangle\rangle=g^{ij}\frac{\partial^{ap} u}{\partial x_i}\frac{\partial^{ap}}{\partial x_j} \langle\nabla h_1,\nabla h_2\rangle=
g^{ij}\frac{\partial u}{\partial x_i}\frac{\partial^{ap}}{\partial x_j}( g^{kl}\frac{\partial h_1}{\partial x_k}\frac{\partial h_2}{\partial x_l})
\end{equation}
and do the same for the other two terms in the RHS of  \eqref{hess-1}.

Using that $g^{ij}=\langle \nabla x_i,\nabla x_j\rangle$ and  $\frac{\partial}{\partial x_i}=\sum_jg_{ij}\nabla x_j$ a standard computation shows that
for any $u\in  \tilde {\mathcal A}$ it holds that
\begin{equation}\label{laplace-coords}
\frac{1}{\sqrt{|g|}}\frac{\partial^{ap}}{\partial x_j}\bigl( g^{jk}\sqrt{|g|}\frac{\partial u}{\partial x_k}\bigr)=Hu(x_i,x_j)g_{ij}
% \Hess u(\nabla x_i,\nabla x_j)g_{ij}=\Tr \Hess u
\end{equation}
on $B_{\delta}(p)$.

%Note that in this computation we are justified in applying the Leibniz rule \eqref{equ:leibniz} whenever necessary.
% because of the following fact.
%Given two $BV\cap C^0$ functions $h_1,h_2$ on an open subset of $\R^n$ it holds that
%\[
%\frac{\partial (h_1h_2)}{\partial x_i}=h_1\frac{\partial h_2}{\partial x_i}+h_2 \frac{\partial h_1}{\partial x_i}
%\]
 %both as measures everywhere  ~\cite[Section 4, Lemma]{Per-DC} and as functions at points of approximate differentiability of $h_1,h_2$.
 
The easiest way to verify formula \eqref{laplace-coords} is as follows.  
%Recall that since $u$ is $DC$ on $V$ it satisfies
 % \[
  %\frac{\partial}{\partial x^i}( \frac{\partial u}{\partial x^j})=  \frac{\partial}{\partial x^j}( \frac{\partial u}{\partial x^i})
  %\]
 % as measures and hence $ \frac{\partial^{ap}}{\partial x^i}( \frac{\partial u}{\partial x^j})=  \frac{\partial^{ap}}{\partial x^j}( \frac{\partial u}{\partial x^i})$ as approximate derivatives a.e. \cite[Theorem 6.8]{Gar-Evans}.
Let $S$ be the set of points in $V$ where  $\nabla u,g_{ij}$ have approximate derivatives and  $ \frac{\partial^{ap}}{\partial x_i}( \frac{\partial u}{\partial x_j})=  \frac{\partial^{ap}}{\partial x_j}( \frac{\partial u}{\partial x_i})$. {Then by~\eqref{2n-der-commute} $S$ has full measure in $V$,
%{\color{red} since $ \frac{\partial^2 u}{\partial x_i\partial x_j}=\frac{\partial^2 u}{\partial x_j\partial x_i}$ as distributions. }
and hence it's enough to verify \eqref{laplace-coords} pointwise on $S$.}

  Let $q\in S$.
  Let $\hat g$ be a smooth metric on a neighborhood of $q$ which such that $\hat g(q)= g(q)$ and $D\hat g(q)=D^{ap} g(q)$.  %where $D^{ap}g(q)$ is the approximate derivative.
   Likewise let $\hat u$ be a smooth function on a neighborhood of $q$ such that $\hat u(q)=u(q), D \hat u(q)=D u(q)$ and $D\frac{\partial \hat u}{\partial x_i}(q)=D^{ap}\frac{\partial  u}{\partial x_i}(q)$ for all $i$. Such $\hat u$ exists (we can take it to be quadratic in $x$) since $ \frac{\partial^{ap}}{\partial x_i}( \frac{\partial u}{\partial x_i})(q)=  \frac{\partial^{ap}}{\partial x_j}( \frac{\partial u}{\partial x_i})(q)$. 
  Then
 
 \[
 \frac{1}{\sqrt{|g|}}\frac{\partial^{ap}}{\partial x_j}\bigl( g^{jk}\sqrt{|g|}\frac{\partial u}{\partial x_k}\bigr)(q)= \frac{1}{\sqrt{|\hat g|}}\frac{\partial}{\partial x_j}\bigl(  \hat g^{jk}\sqrt{|\hat g|}\frac{\partial \hat u}{\partial x_k}\bigr)(q)
 \]
where all  the derivatives are approximate derivatives.

Similarly
\[
H u(x_i, x_j)(q)g_{ij}(q)=H \hat u(x_i, x_j)(q)\hat g_{ij}(q)
\]
where again all  the derivatives in  \eqref{hess-1}  and \eqref{hess-2}   are approximate derivatives.

But 
\[
\frac{1}{\sqrt{|\hat g|}}\frac{\partial}{\partial x_j}\bigl( \hat g^{jk}\sqrt{|\hat g|}\frac{\partial \hat u}{\partial x_k}\bigr)(q)=\Hess_{\hat g} \hat u(\nabla_{\hat g}  x_i,\nabla_{\hat g}  x_j)\hat g_{ij}(q)
\]
by standard Riemannian geometry since all functions involved are smooth. Since $q\in S$ was arbitrary this proves that \eqref{laplace-coords} holds a.e. in the sense of approximate derivatives as claimed.

\medskip
{\bf 4.}
It follows that 
\begin{align}\label{equ:A}
&\Tr \Hess (\chi u) |_{B_{\delta}(p)}= \Hess (\chi u)(\nabla x_i, \nabla x_j) g_{ij} |_{B_{\delta}(p)}\smallskip\nonumber\\
&\ \ \ \ \ \ \ \ \ \ \ = \Hess (\chi u)(\nabla \chi x_i, \nabla \chi x_j) g_{ij} |_{B_{\delta}(p)}= H u(x_i,x_j)g_{ij}=\Delta_0 u|_{B_\delta(p)}.
\end{align}
for every $u\in \tilde{\mathcal{A}}$ where $\Hess$ is the Hessian in the sense of Gigli, and $H(u)$ is denotes the RHS of (\ref{hess-1}).
%\[
 %\Delta_0^a=\Tr \Hess\quad \text{ on }  \tilde {\mathcal A}
%\]
The first equality in \eqref{equ:A} is the definition of $\Tr$, the second equality is the $L^{\infty}$-homogeneity of the tensor $\Hess (\chi u)$, and the third equality is the identity (\ref{hess-1}).

\begin{comment}{\color{blue}
Recall that $f$ is locally Lipschitz. Then, since we assume $X=\supp\m$, the set $\{x\in X: f(x)=0\}$ has $\mathcal{H}^n$-measure $0$. 
Hence, we can pick $p\in X^g_{reg}$ and $\epsilon>0$ as before such that $B_{2\epsilon}(p)\subset \{x\in X: f(x)>0\}$.}
\end{comment}

Since $f$ is locally Lipschitz { and positive on $B_{\delta}(p)$}, we can perform the following integration by parts in $DC_0$ coordinates. Let $u\in \tilde{\mathcal{A}}$ and let $g$ be Lipschitz with compact support in $B_\delta(p)$. $\chi u\in D_{L^2(\m)}(\Delta)$ implies 
$u|_{B_\delta(p)}\in D({\bf \Delta},B_{\delta}(p))$. Then
\begin{align*}
\int_{B_{\delta}(p)} g \Delta u d\m&=-\int_{B_{\delta}(p)} \langle \nabla u, \nabla g\rangle d\m=-\int_{B_{\delta}(p)} \langle \nabla u, \nabla g\rangle f d\mathcal{H}^n\\
&= - \int_{B_{\delta}(p)} \langle \nabla u,\nabla (g f)\rangle d\mathcal{H}^n +\int_{B_{\delta}(p)} \langle \nabla u, \nabla \log f\rangle g f d\mathcal{H}^n\\
&= \int_{B_{\delta}(p)} (\Delta_0 u +\langle \nabla u,\nabla \log f\rangle) g d\m 
\end{align*}
 yields $$\Delta u=\Delta_0 u+\langle \nabla u, \nabla \log f\rangle$$ on $B_{\delta}(p)$ for any $u \in \tilde {\mathcal A}$. Note again that only $\chi u$ is in $D_{L^2(\m)}(\Delta)$.
%\[
 %\Delta^a=\Tr \Hess-\langle \cdot, \nabla f\rangle\quad \text{ on }  \tilde {\mathcal A}
%\]

On the other hand, by Corollary \ref{cor:trace} it holds that $\Delta (\chi u)=\Tr \Hess (\chi u)$ $\m$-a.e.\ .
Thus $$0=
\Tr \Hess (\chi u) |_{B_{\delta}(p)}-
\Tr \Hess (\chi u) |_{B_{\delta}(p)}=\Delta u|_{B_\delta(p)}-\Delta_0 (\chi u)|_{B_\delta(p)}=\langle \nabla u,\nabla \log f\rangle|_{B_{\delta}(p)}$$ a.e. for any $u\in \tilde {\mathcal A}$. 

\medskip
{\bf 5.} Therefore  { $f \nabla \log f|_{B_\delta(p)}=\nabla f|_{B_\delta(p)}=0$.}
Indeed, since $f$ is semiconcave, $f\circ x^{-1}$ is DC by ~\cite{Lytchak-Nagano18}. Hence $\nabla f=g^{ij}\frac{\partial f}{\partial x_i}$ is continuous on a set of full measure $Z$ in $B_\delta(p)$ since this is true for convex functions on $\R^n$. Let $q\in Z$ be a point of continuity of  $\nabla f|_Z$ and $v=\nabla f(q)$. Assume $v\ne 0$. Then due to extendability of geodesics there exists $z\notin U$ such that $\nabla d_z(q)=\frac{v}{|v|}$. Since  $\nabla d_z$ is continuous near $q$ and $\nabla f$ is continuous on $Z$ it follows $\langle \nabla f,\nabla d_z\rangle\ne 0$ on a set of positive measure.
Hence $\nabla f|_{B_\delta(p)}=0$ and $f|_{B_\delta(p)}=const$. 
%{\color{blue} Thus, $f$ is locally constant on the open set $X^g_{reg}\cap \{f> 0\}$. }

\medskip
{\bf 6.}
We claim that this implies that $f$ is constant on { $X^g_{reg}$}. (This is not immediate since we don't know yet that { $X^g_{reg}$} is connected.)
Indeed,
since $X$ is essentially nonbranching,  radial disintegration of $m$ centered at $p$ ({Theorem~\ref{th:laplacecomparison}}) %(Remark \ref{rem:disintegration})
 implies that for almost all $q\in X$ the set $[pq]\cap X^g_{reg}$ has full measure in $[p,q]$. It is also open in $[p,q]$ since $X^g_{reg}$ is open.

Suppose $q\in X^g_{reg}$ is as above.

Since $\theta$ is semiconcave {on $X$} and locally constant on   $X^g_{reg}$ it is locally Lipschitz {(and hence Lipschitz)} on the geodesic segment $[p,q]$. A Lipschitz function on $[0,1]$ which is locally constant on an open set of full $\mathcal L^1$ measure is constant. Therefore $\theta$ is constant on $[p,q]$ and hence $\theta$ is constant on  $X^g_{reg}$  which has full measure.
 Therefore $f=\theta=const$ a.e. globally. \qed
\small{
\bibliographystyle{amsalpha}
\bibliography{new}

}
\end{document}